\documentclass[12pt]{amsart}
\theoremstyle{plain}
\usepackage{fullpage}
\usepackage{tikz}
\newtheorem{thm}{Theorem}[section]

\newtheorem{lemma}[thm]{Lemma}
\newtheorem{cor}[thm]{Corollary}

\theoremstyle{definition}
\newtheorem{defn}[thm]{Definition}
\theoremstyle{remark}
\newtheorem{remark}{Remark}
\newtheorem{example}{Example}
\usepackage{amsmath}
\usepackage{amsthm}
\usepackage{amsfonts}
\usepackage{amssymb}
\usepackage{amscd}
\usepackage{graphicx}
\usepackage[all]{xy}
\usepackage{amsmath}
\usepackage{amsfonts}
\usepackage{latexsym}
\numberwithin{thm}{section}

\usepackage{verbatim}
\input amssym.def
\input amssym.tex
\usepackage{hyperref}
\usepackage{url}
\usepackage{color}
\usepackage{breakurl}

\newcommand\N{\mathbb{N}}
\newcommand\Z{\mathbb{Z}}

\numberwithin{equation}{section}

\begin{document}
\title{Enumeration of k-Exceedance Lattice Paths with an Application to Comparing Chains of Order Statistics}
\author{Charles Hoffman and Corey Manack}
\maketitle
\begin{abstract}
We enumerate the number of monotonic lattice paths starting at $(0,0)$ and terminating at $(m,n)$ in which $l$ of the first $k$ steps lie below the line $y=x$ (equation \eqref{equ:kexceedancepaths}). This combinatorial formula is a summation whose terms are a product of Catalan numbers, ballot numbers and binomial coefficients. We apply these combinatorial formulas to failure analysis by deriving a probability distribution that compares the performance of a $k$-out-of-$m$ system to a $k$-out-of-$n$ system of continuous, independent and identically distributed random variables. 
\end{abstract}
\section{Introduction}
Lattice path combinatorics have a variety of applications: ballot numbers and its variants \cite{HN2011)}, moments of orthogonal polynomials \cite{HN2012)}, Smirnov test statistics, knock out tournaments and convolution identities (see, e.g. \cite{Nar79})), to name a few. The statistical applications arise from comparing collections of independent and identically distributed random variables. In this paper, we enumerate so-called $k$-exceedant paths and produce a probability distribution to compare the failure rates between two parallel systems whose components consist of independent and identically distributed random variables. We enumerate $k$-exceedant paths (definition \ref{defn:exceedant}) in $\S\ \ref{sec:latticepaths}$ then derive the probability distribution for order statistics in $\S\ \ref{sec:chainsofstats}$. We provide an interpretation of this model to random walks in $\S\ \ref{sec:rws}$ and lastly give asymptotics of  equation \eqref{equ:kexceedancepaths} in $\S\ \ref{sec:asymptotics}$. 
\section{Lattice paths}  
\label{sec:latticepaths}
A {\emph{lattice path}} in $\N^2$ is a sequence of points which start at $(0,0)$ and take steps in the $(1,0)$ (up) and $(0,1)$ (right) directions. Denote by $\Gamma_{m,n}$ the set of lattice paths from $(0,0)$ to $(m,n)$. 
\begin{defn}
\label{defn:exceedant}
A path $\gamma \in \Gamma_{m,n}$ has $\textbf{horizontal exceedance}$ $l$ if $\gamma$ has $l$ horizontal edges below the line $y=x$, $\textbf{vertical exceedance}$ $l$ if $\gamma$ has $l$ vertical edges above the line $y=x$, $\textbf{$k$-horizontal exceedance}$ equal to $l$ if $l$ of the first $k$ horizontal edges lie below the line $y = x$ and $\textbf{$k$-vertical exceedance}$ $l$ if $\gamma$ has $l$ of its first $k$ vertical edges lie above the line $y=x$.    
\end{defn}
We write $HE(\gamma)$, $VE(\gamma)$, $HE_k(\gamma)$, $VE_k(\gamma)$ for the horizontal, vertical, $k$-horizontal and $k$-vertical exceedance of $\gamma$ respectively. 
\begin{example} 
The path $\gamma = RUURRRRURU\in \Gamma_{6,4}$ has horizontal exceedance $5$ and vertical exceedance $1$. If we set $k = 4$, then $\gamma$ has $4$-horizontal exceedance $3$ and $4$-vertical exceedance $1$. The $4$-exceedances of $\gamma$ are depicted in Figure \ref{Fig2}. 
\begin{center}
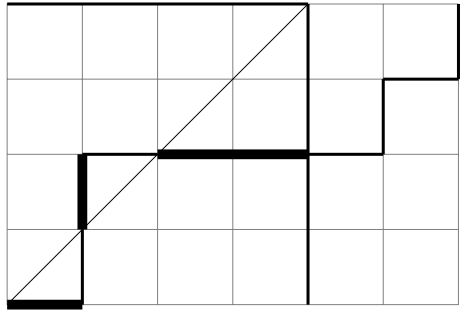
\begin{figure}[h!]
\begin{tikzpicture}
\draw[help lines] (0,0) grid (6,4);
\draw [line width=1.3mm] (0,0) -- (1,0);
\draw [line width=.4mm] (1,0) -- (1,1);
\draw [line width=1.3mm] (1,1) -- (1,2);
\draw [line width=.4mm] (1,2) -- (2,2);
\draw [line width=1.3mm] (2,2) -- (4,2);
\draw [line width=.4mm] (4,2) -- (5,2);
\draw [line width=.4mm] (5,2) -- (5,3);
\draw [line width=.4mm] (5,3) -- (6,3);
\draw [line width=.4mm] (6,3) -- (6,4);
\draw [line width=.4mm] (0,4) -- (4,4);
\draw [line width=.4mm] (4,0) -- (4,4);
\draw [line width=.1mm] (0,0) -- (4,4);
\end{tikzpicture}
\caption{$HE_4(\gamma)=3$, $VE_4(\gamma)=1$.}
\label{Fig2}
\end{figure}
\end{center}
\end{example}
We focus on enumerating the exceedant $l$ and $k$-exceedant $l$ paths $\gamma\in \Gamma_{m,n}$. We write $\# (x,y,l)$ for the number of paths $\eta\in \Gamma_{x,y}$ satisfying $HE(\eta)=l$. Clearly, $\# (x,y,l) = 0$ if $l > x$. When $x=y$, the Chung-Feller theorem \cite{CF1949} states that lattice paths terminating at $(i,i)$ are partitoned uniformly among all possible exceedances. In particular, $\#(i,i,l)=C_i$ the $i$th Catalan number. The next lemma combines the previous two facts. 
\begin{thm}[Chung-Feller]
\label{Chung-Feller}
\begin{equation}
\label{y=x}
\# (i,i,l) = \left\{
        \begin{array}{ll}
            C_i =\frac{1}{i+1}{2i\choose i} & \quad 0\leq i \leq l \\
            0 & \quad \text{otherwise} \\
        \end{array}
    \right.
\end{equation}
\end{thm}
\noindent A bijective proof of the Chung-Feller theorem can be found in \cite{GS2003}.\\
\begin{lemma}
\label{y<x}
For $\gamma \in \Gamma_{x,y}$ and $0\leq k \leq \min\{x,y\}$,
\begin{align}
HE_k(\gamma) &+ VE_k(\gamma) = k. 
\end{align}
\end{lemma}
\begin{proof}
Observe that $k$-exceedance of a path $\gamma\in\Gamma_{x,y}$ is determined by the segment whose terminal vertex lies on the line $y=k$ or $x=k$. Such a segment extends uniquely to a path $\tilde{\gamma}\in \Gamma_{k,k}$ by appending edges from this terminal vertex to $(k,k)$. So $HE_k(\gamma) = HE_k(\tilde{\gamma})$ and $VE_k(\gamma) = VE_k(\tilde{\gamma})$. By monotonicity, $\tilde{\gamma}$ has $2l$ edges below the line $y =x$, and each path in $\Gamma_{k,k}$ has exactly $2k$ edges, so $\tilde{\gamma}$ has $2(k-l)$ edges above $y=x$. Hence $VE_k(\tilde{\gamma}) = k-l$, and
\[HE_k(\gamma) + VE_k(\gamma) = HE_k(\tilde{\gamma}) + VE_k(\tilde{\gamma}) =  k\]
which proves the lemma.
\end{proof}
\begin{cor}
\label{y<x2}
For $\gamma \in \Gamma_{x,y}$. 
\begin{align}
HE(\gamma) &+ VE(\gamma) = \max\{x,y\}.
\end{align}
\end{cor}
\begin{proof}
Notice that $\gamma \in \Gamma_{x,y}$ extends uniquely to a path $\tilde{\gamma} \in \Gamma_{\max\{x,y\},\max\{x,y\}}$ without affecting exceedances. Setting $k = \max\{x,y\}$, the corollary follows from lemma \ref{y<x}.   
\end{proof}
Combining lemma \ref{y<x} and corollary \ref{y<x2} reduces the task of computing $\#(x,y,l)$ to the case $y<x$.
Reflecting a path $\gamma\in \Gamma_{m,n}$ through the line $y=x$ and applying corollary \ref{y<x2}, yields the identity
\begin{align}
\label{reflect}
\# (m,n,l) = \# (n,m,\max\{m,n\} - l)
\end{align}
which reduces $\# (m,n,l)$ to the case $n<m$. 
\begin{defn}
The ballot number $b(m,n)$ is the number of lattice paths from $(0,0)$ to $(m,n)$ that stay strictly below the line $y=x$, except at $(0,0)$, and has many combinatorial interpretations. $b(m,n)$ is given by the formula
\begin{equation}
\label{avoid}
b(m,n) = {{m+n-1}\choose {n}} - {{m+n-1}\choose {n-1}} = \frac{m-n}{m+n}{{m+n}\choose n}
\end{equation}
\end{defn}
\begin{thm}
\label{thm:exceedant}
Fix a point $(x,y)\in\N^2$ and suppose $x>y$. As a function of $l$,
\begin{equation}
\label{equ:exceedant}   
\#(x,y,l) = \left\{
     \begin{array}{lr}
       \sum_{i= x-l}^yC_i b(x-i,y-i) & \text{if $x-y \leq l \leq x$}\\
       0 &  \text{otherwise}
     \end{array}
   \right. 
\end{equation}
\end{thm}
\begin{proof}
For fixed $l\in \N$, partition $A = \{\gamma \in \Gamma_{x,y} | HE(\gamma) = l\}$ into the (possibly empty) sets $A_0, A_1,\ldots, A_y$ where $A_i$ is the set of lattice paths $\gamma$ that pass the line $y=x$ at the point $(i,i)$. To ensure $A$ is nonempty, $l\leq x$, and the assumption $x>y$ guarantees a lattice path which terminates at $(x,y)$ has horizontal exceedance at least $x-y$,
\[\#(x,y,l) = 0, \text{ if $x>l$ or $l<x-y$}\]
Because $x>y$, the terminal segment $\gamma_t$ of $A_i$ from $(i,i)$ to $(x,y)$ must lie below the line $y=x$. Thus, the terminal segment of $\gamma\in A_i$ contributes $HE(\gamma_t)=(x-i)$ so $HE(\gamma_i)= l-(x-i)$ and to ensure $A_i$ is nonempty,
\[l-(x-i) \geq 0 \text{ or } i \geq x-l\]
thus nonempty $A_i$ run from $i= x-l$ to $i = y$. We call such indices admissible. To count nonempty $A_i$, we multiply the number of initial segments by the number of terminal segments as described above. By lemma \ref{Chung-Feller}, the number initial segments terminating at $(i,i)$ of horizontal exceedance $l-(x-i)$ is
\[C_i=\frac{1}{i+1}{{2i}\choose i}\]
By equation \ref{avoid}, the number of terminal (monotonic) paths from $(i,i)$ to $(x,y)$ that stay below the line $y=x$ is given by
\[b(x-i,y-i)= \frac{x-y}{x+y-2i}{{x+y-2i}\choose{y-i}}\]
making
\[|A_i| =C_ib(x-i,y-i).\]
Summing over admissible $i$ completes the theorem.
\end{proof}
\begin{remark}
By deleting the terminal edge, one can see that equation \ref{thm:exceedant} satisfies the necessary recursion relation
\[\#(x,y,l) = \#(x-1,y,l-1) + \#(x,y-1,l)\]
at those points $(x,y)$ with $x>y$.
\end{remark}
We have reached the main theorem, which provides a closed form solution to $\#(m,n,k,l)$, the number of lattice paths terminating at $(m,n)$ having $k$-horizontal exceedance $l$. 
\begin{thm}
\label{thm:kexceedancepaths}
\begin{equation}
\label{equ:kexceedancepaths}
\#(m,n,k,l)=\sum_{j=0}^{k-1}I(\{j\geq l\})|B_{(j,k)}(l)| + \sum_{j=0}^{k-1} I(\{j\geq k-l\})|B_{(k,j)}(l)|
\end{equation} where
\begin{align*}
|B_{(k,j)}(l)|    &=  {{m+n-k-j}\choose {m-k}}\sum_{i= k-l}^jC_i b(k-1-i,j-i)\text{,  $0\leq j < k-1$}\\
|B_{(j,k)}(l)|    &=  {{m+n-k-j}\choose {n-k}}\sum_{i= l}^jC_i b(k-1-i,j-i)\text{, $0\leq i < k-1$}\\
|B_{(k-1,k)}(l)| &=  {{m+n-2k+1}\choose{n-k}}C_{k-1}\\
|B_{(k,k-1)}(l)| &= {{m+n-2k+1}\choose{m-k}}C_{k-1}
\end{align*}
and $I$ is the indicator function.
\end{thm}
\begin{proof}
The proof is analogous to theorem \ref{thm:exceedant}; we begin by fixing $m,n,k,l\in\N$. As stated before, $k$-horizontal exceedance of $\gamma\in \Gamma_{m,n}$ is determined once its initial segment reaches either $x=k$ or $y=k$. Invoke theorem \ref{thm:exceedant} on each initial segment to partition 
\[B = \{\gamma \in \Gamma_{m,n} | HE_k(\gamma) = l\}\] into the (possibly empty) sets \[B_{(0,k)}, B_{(1,k)},\ldots, B_{(k-1,k)},B_{(k,k-1)},\ldots,B_{(k,1)},B_{(k,0)},\] where $B_{(j,k)}$ (resp. $B_{(k,j)}$) is the set of $\gamma\in\Gamma_{(m,n)}$ with $HE_k(\gamma) = l$ and initial segment which first crosses the line $y=k$ (resp. $x=k$) at the point $(j,k)$ (resp $(k,j)$). Notice that $k$-horizontal exceedance of $\gamma\in B_{(k,j)}$ is determined by the the terminal vertex $(k,j)$ of the initial segment. As $j<k$, this initial segment has its last edge below the line $y=x$ and by definition of $B_{(k,j)}$ this edge must be horizontal. Since this last edge contributes $1$ to the horizontal exceedance, $\#(k-1,j,l-1)$ counts the number of such initial segments in $B_{(k,j)}$. This argument holds except at $j=k-1$, where $\#(k-1,k-1,l-1)=C_{k-1}$ by lemma \ref{Chung-Feller}. Once $k$-exceedance is determined, we may append to the initial segment any of the 
\[{{m+n-k-j}\choose {m-k}}\]
 lattice paths from $(k,j)$ to $(m,n)$ to construct a path in $B_{(k,j)}$. Multiplying, we get for $j$ satisfying $k-l\leq j < k-1$,
\begin{align*}
|B_{(k,j)}|  &=   \#(k-1,j,l-1){{m+n-k-j}\choose {m-k}}\\
                   &=  {{m+n-k-j}\choose {m-k}}\sum_{i= k-l}^jC_i b(k-1-i,j-i)\text{, and}\\
|B_{(k,k-1)}| &= {{m+n-2k+1}\choose{m-k}}C_{k-1}
\end{align*}
Counting paths in $B_{(j,k)}$ is slightly different since $\gamma\in B_{(j,k)}$ has an initial segment $\gamma_i$ which determines $k$-horizontal exceedance with vertical terminal edge {\emph{above}} the line $y=x$. Since this terminal edge does not contribute to horizontal exceedance, $\#(j,k-1,l)$ counts the number of such initial segments in $B_{(j,k)}$. As before, we may append to $\gamma_i$ any of the 
\[{{m+n-k-j}\choose {n-k}}\] lattice paths from $(j,k)$ to $(m,n)$, to
construct a path in $B_{(j,k)}$.
Recalling the the reflection identity \eqref{reflect} we have, for $l\leq j < k-1,$
\begin{align*}
|B_{(j,k)}|  &=   \#(k-1,j,k-1-l){{m+n-k-j}\choose {n-k}}\\
                   &=  {{m+n-k-j}\choose {n-k}}\sum_{i= l}^jC_i b(k-1-i,j-i)\text{, and}\\
|B_{(k-1,k)}| &= {{m+n-2k+1}\choose{n-k}}C_{k-1}
\end{align*}
It remains to show which of the sets \[B_{(0,k)}, B_{(1,k)},\ldots, B_{(k-1,k)},B_{(k,k-1)},\ldots,B_{(k,1)},B_{(k,0)},\] are nonempty for a particular $l$. We separate $l$ into cases:\\
Case $l=0$: No horizontal lines lie below $y=x$ in the $k\times k$ box, so the nonempty sets are
\[B_{(0,k)},\ldots,B_{(k-1,k)}.\]
which sum over $l$, $0\leq l\leq k-1$ to get 
\begin{align*}
\#(m,n,k,0) &= \left(\sum_{j=0}^{k-1}|B_{(j,k)}|\right)\\ 
\end{align*}
Case $l=k$: All horizontal lines lie below $y=x$ in the $k\times k$ box, so the nonempty sets are
\[B_{(k,0)},\ldots ,B_{(k,k-1)}.\]
which sum over $l$, $0\leq l\leq k-1$ to get 
\begin{align*}
\#(m,n,k,k) &= \left(\sum_{j=0}^{k-1}|B_{(k,j)}|\right)\\ 
\end{align*}
In the case $0<l<k$, the condition $HE_k(\gamma) = l$ forces nonemptiness only in the sets  
\[B_{(k,k-l)},\ldots ,B_{(k,k-1)},B_{(k-1,k)},\ldots,B_{(l,k)}.\]
Which we sum to get 
\begin{align*}
\#(m,n,k,l) &= \sum_{j = k-l}^{k-1}|B_{(k,j)}| + \sum_{j=l}^{k-1}|B_{(j,k)}| 
\end{align*}
All cases align precisely with formula (\ref{equ:kexceedancepaths}) as written.
\end{proof}
\section{Chains of order statistics}
\label{sec:chainsofstats}
We now give an application of theorems \ref{thm:exceedant}, \ref{thm:kexceedancepaths} to failure analysis. Typically, failure rates between $k$-out-of-$n$ systems are compared by comparing their respective $k$th order statistics \cite{D1997}, \cite{KK2000}. This requires an understanding of the underlying component distributions. We focus on comparing the performance of a $k$-out-of-$m$ system $\{X_1,\ldots,X_m\}$ to a $k$-out-of-$n$ system $\{Y_1,\ldots,Y_n\}$, where all components have continuous, independent, and identical distributions. This model compares the failure rates between the two systems by fixing $k$ and comparing the $l$th order statistics for each $l$ between $1$ and $k$. This model relies solely the number of components ($m,n$) and {\emph{not}} the underlying distributions. This nonparametic model yields a new probability distribution for comparing collections of order statistics. Denote by $X^{(k)}$ the $k$th order statistic of $X_1,\ldots,X_n$, i.e., $X^{(k)}=\min_k\{X_1,\ldots,X_n\}$ is the $k$th smallest value of $\{X_1,\ldots, X_n\}$. Obviously, $k$ is meaningful only when $0<k\leq n$. The following lemma simply conveys the exchangeablility of continuous $i.i.d$ random variables for order statistics.
\begin{lemma}
\label{uniform}
Let $X_1,\ldots,X_n$ be a sequence of continuous, independent and identically distributed (c.i.i.d.) random variables, and let $X^{(i)}$ be the $i$-th order statistic. For $1\leq i,j \leq n$,  $\Pr[X^{(i)} = X_j ] = 1/n$.
\end{lemma}
\begin{proof}
For a fixed index $i$, $X^{(i)}$ must be one of the random variables $X_1,\ldots,X_n$. By continuity and independence of $X_1,\ldots,X_n$, $\Pr[X_j = X_k] = 0$ over all such pairs $j\neq k$. Hence,
\[\Pr[X^{(i)} = X_1] + \cdots + \Pr[X^{(i)} = X_n] = 1.\] 
That $X_1,\ldots,X_n$ are identically distibuted implies equality between any pair of summands. The lemma follows. 
\end{proof}
\begin{remark}
The assumption of continuity in Lemma \ref{uniform} is essential since $\Pr[X_i = X_j]>0$ is possible if $X_i, X_j$ were atomic.
\end{remark}
\begin{remark} 
Under the assumption of continuity, we may drop the case of equality when comparing $X_i$'s, since
\[\Pr\left(\bigcup_{i,j} \{X_i = X_j\}\right) = 0.\]
\end{remark}
Now, for integers $m\leq n$, let $X_1,\ldots, X_m, Y_1,\ldots Y_n$ be c.i.i.d. random variables mapping into a totally ordered space $(\Omega, <)$. We seek a probability distribution that compares the bottom $k$ performers in each of the systems $X = \{X_1,\ldots, X_m\}, Y= \{Y_1,\ldots Y_n\}$. Clearly $k\leq m$. \\
Our method of comparing chains of order statistics is induced from the order on $(\Omega, <)$. 
\begin{defn}
For $c.i.i.d$ random variables $X_1,X_2,\ldots,X_m$ mapping into $(\Omega, <)$, an (ascending) {\textbf{chain}} of length $d$ is the ordering
\[X_{i_1} < X_{i_2} < \cdots < X_{i_d}\] such that $X_{i_j} = X^{(j)}$ for each $1\leq j \leq d$. 
\end{defn}
\begin{defn}
\label{defn:order}
 Denote by $X^{((m))}$ the chain of order statistics
\[X^{(1)} < X^{(2)} < \cdots < X^{(m)}.\] 
Fix fixed $k$ and $l$ between $1$ and $k$, define the ordering $<_{k,l}$ on chains by the following:
\[ X^{((m))} <_{k,l} Y^{((n))}\]
if and only if $\left| \left\{i\mid 1\leq i\leq k \text{ and } X^{(i)} < Y^{(i)}\right\} \right| = l. $
\end{defn}    
Definition \ref{defn:order} says that $X^{((m))} <_{k,l} Y^{((n))}$ if, of the $k$ bottom performers, there are exactly $l$ instances when the $i$-th bottom performer from system $X$ underperformed the $i$-th bottom performer from system $Y$. The main result of this section (theorem \ref{thm:chainstopaths}) is to determine $\Pr\left(X^{((m))} <_{k,l} Y^{((n))} \right)$ for integers $m,n,k,l$ which satisfy $0\leq l\leq k\leq m\leq n$ 
\begin{thm}
\label{thm:chainstopaths}
\begin{equation}
\label{thm:chains}
\Pr\left(X^{((m))} <_{k,l} Y^{((n))} \right) = \frac{\#(m,n,k,l)}{{{m+n}\choose m}}  
\end{equation}
where $\#(m,n,k,l)$ counts the number of lattice paths from $(0,0)$ to $(m,n)$ of $k$-exceedance $l$ as in theorem \ref{equ:kexceedancepaths}.
\end{thm}
The proof of theorem \ref{thm:chainstopaths} comes from a correspondence between chains and lattice paths. Define $B_{m,n}$ to be the set of ascending chains of length $m+n$ from \[\{X_1,\ldots,X_m,Y_1,\ldots,Y_n\}\] and $C_{m,n}$ to be the set of ascending chains of length $m+n$ from the set \[\{X^{(1)},\ldots,X^{(m)},Y^{(1)},\ldots,Y^{(n)}\}.\] Each chain $c\in C_{m,n}$ corresponds uniquely to a pair of chains $X^{((m))}$, $Y^{((n))}$ and for any admissible $k$, one can determine which $l\in\N$ satisfies $X^{((m))} <_{k,l} Y^{((n))}$.
\begin{example}
For $k=4$, the chain
\[X^{(1)} < Y^{(1)} < Y^{(2)} < X^{(2)} < X^{(3)} < X^{(4)} < X^{(5)} < Y^{(3)} < X^{(6)} < Y^{(4)}\]
in $C_{6,4}$ satisfies
\[X^{(1)} < Y^{(1)}\]
\[X^{(2)} > Y^{(2)}\]
\[X^{(3)} < Y^{(3)}\]
\[X^{(4)} < Y^{(4)}\]
and thus satisfies $X^{((6))} <_{4,3} Y^{((4))}$.
\end{example}
\begin{lemma}
\label{thm:uniformchains}
Let  $X_1,\ldots,X_m,Y_1,\ldots,Y_n $ be a sequence c.i.i.d. random variables. Then the probability distribution on $C_{m,n}$ is uniform, with $$\Pr(c) = \frac{1}{{{m+n}\choose m}} = \frac{m!n!}{(m+n)!},\ \  c\in C_{m,n}$$
\end{lemma}
\begin{proof}
Lemma \ref{uniform} reduces the argument to counting chains in $B_{m,n}$. Observe, $C_{m,n}$ is the set of equivalence classes in $B_{m,n}$ under the action of $S_m \times S_n$ where the $S_m$ factor permutes $X_i$'s, and $S_n$ permutes $Y_j$'s. Therefore, each $c\in C_{m,n}$ contains $m!n!$ elements; it is clear $|B_{m,n}| = (m+n)!$. The lemma follows. 
\end{proof}
Thus, we have reduced the task of computing 
\[\Pr\left(X^{((m))} <_{k,l} Y^{((n))} \right)\]
to counting those chains in $C_{m,n}$ satisfying $X^{((m))} <_{k,l} Y^{((n))}$. Observe, $C_{m,n}$ is in bijective correspondence $\Gamma_{m,n}$; this can be seen by constructing a lattice path from reading a chain in ascending order, taking a north step when a $Y^{(i)}$ is encountered and an east step for$X^{(i)}$. Conversely, a chain can be constructed from $\gamma\in \Gamma_{m,n}$ by writing $X^{(i)}$ when the $i$th horizontal edge is encountered, and $Y^{(i)}$ when the $i$th vertical edge is encountered. Under this bijection, it is clear \[|C_{m,n}| = |\Gamma_{m,n}| = {{m+n}\choose m}.\]
\begin{example} 
\label{path} The chain
\[X^{(1)} < Y^{(1)} < Y^{(2)} < X^{(2)} < X^{(3)} < X^{(4)} < X^{(5)} < Y^{(3)} < X^{(6)} < Y^{(4)}\]
in $C_{6,4}$ corresponds to the path RUURRRRURU in Example 1.
\end{example} 
We are now ready to prove the main theorem.
\begin{proof}
Proof of theorem \ref{thm:chainstopaths}: As noted above, each $c \in C_{m,n}$ corresponds uniquely to some path $\gamma \in \Gamma_{m,n}$, and there is an integer $l$ for which 
\[ X^{((m))} <_{k,l} Y^{((n))}\] is satisfied by $c$. Then, for exactly $l$ indices between $1$ and $k$, $X^{(j)} < Y^{(j)}$. Equivalently stated, there are $l$ instances when $X^{(j)}$ is appended to an initial chain that contains at least as many $X^{(i)}$'s as $Y^{(i)}$'s. The corresponding statement in $\Gamma_{m,n}$ says there are exactly $l$ instances where the $j$-th horizontal edge appears before the $j$-th vertical edge in the initial segment of $\gamma$ lying in the $k\times k$ box. It is at these indices where a horizontal edge is appended to a path whose initial segment lies on or below the line $y = x$, and this happens exactly $l$ times inside the $k\times k$ box. Thus $HE_k(\gamma) = l$.  We have shown \[\Pr(\{c\in C_{m,n} | X^{((m))} <_{k,l}  Y^{((n))}\}) = \Pr(\{\gamma \in \Gamma_{m,n} | HE_k(\gamma) = l\}).\] The proof follows from theorem \ref{thm:uniformchains}.
\end{proof}
\begin{cor}
Let  $X_1,\ldots,X_m,Y_1,\ldots,Y_n $ be a sequence of c.i.i.d. random variables. Then, for fixed $k$,
\[\Pr(\{c\in C_{n,m} | X^{((m))} >_{k,l} Y^{((n))}\}) = \Pr(\{\gamma \in \Gamma_{m,n} | VE_k(\gamma) = l\})\]
\end{cor}
\begin{proof}
Follows by reflecting through the line $y=x$, and applying theorem \ref{thm:chainstopaths}.
\end{proof}
\section{Random Walks}
\label{sec:rws}
The failure rate model has an interesting interpretation in terms of random walks on $\Z$. 
The bijection between the sets
\[\{c\in C_{m,n} | X^{((m))} <_{k,l}  Y^{((n))}\} \text{ ,and } \{\gamma \in \Gamma_{m,n} | HE_k(\gamma) = l\}.\]
and the map $$(1,0)\to +1$$
                     $$(0,1)\to -1$$
establishes a bijection between the set $\Gamma_{m,n}$ and the set $W_{m,n}$ of length $m+n$ integer walks on $\Z$ that start at $x=0$ and terminate at $x=m-n$. As mentioned before, by monotonicity in $\Gamma_{m,n}$, any $\gamma$ that satisfies $HE_k(\gamma) = l$, has exactly $2l$ of its first $2k$ steps lying below the main diagonal. In $W_{m,n}$, these paths spend exactly $2l$ of their first $2k$ steps above $x=0$. Thus, if we let $T_{2k}\colon W_{(m,n)}\to \{0,2,\ldots,2k\}$ be the number of initial $2k$ steps lying above $x=0$, we have yet another interpretation of $X^{((m))} <_{k,l}  Y^{((n))}$:
\[\Pr\left(\{c\in C_{m,n} | X^{((m))} <_{k,l}  Y^{((n))}\}\right) = \Pr\left(\{w \in W_{m,n} | T_{2k}(w) = 2l\}\right)\]
\section{Asymptotics}
\label{sec:asymptotics}
We determine the asmptotics of $\#(m,n,k,l)$ for a few special cases. Recall that formula \eqref{equ:kexceedancepaths} given by  
\[\#(m,n,k,l)= \sum_{j=0}^{k-1}I(\{j\geq l\})|B_{(j,k)}(l)| + \sum_{j=0}^{k-1} I(\{j\geq k-l\})|B_{(k,j)}(l)|\]
counts the number of paths $\gamma\in\Gamma_{m,n}$ with $k$-exceedence equal to $l$. One case is to fix $k,l$, and set $m=n$. For $m$ sufficiently large,
\begin{align*}
|B_{(k,j)}(l)|    &\approx {{2m}\choose {m}}\sum_{i= k-l}^jC_i b(k-1-i,j-i) \text{,  $0\leq j < k-1$}\\
|B_{(j,k)}(l)|    &\approx {{2m}\choose {m}}\sum_{i= l}^jC_i b(k-1-i,j-i)\text{, $0\leq i < k-1$}\\
|B_{(k-1,k)}(l)| &\approx  {{2m}\choose{m}}C_{k-1}\\
|B_{(k,k-1)}(l)| &\approx {{2m}\choose{m}}C_{k-1}
\end{align*}
So that $\#(m,m,k,l)\approx {2m\choose m}(B_{m,k,l})$ where $B_{m,k,l}$ is the sum of products of Catalan and ballot numbers given in the four formulas above. The authors see no way to easily reduce the formula $B_{m,k,l}$, so we calculate the a few small cases of $k$ in the table below. For $m$ sufficiently large, $\#(m,m,k,l)$ is given in the table below: 
\begin{center}
\begin{tabular}{c|c|c|c|c}
k & $\#(m,m,k,0)$ & $\#(m,m,k,1)$ & $\#(m,m,k,2)$  \\ \hline
0 & ${{2m}\choose m}$ & & & \\ \hline
1 & ${{2m-1}\choose {m-1}}$ & ${{2m-1}\choose {m-1}}$ &  & \\ \hline
2 & $ {{2m-2}\choose {m}}+2{{2m-3}\choose {m-1}}+2{{m-2}\choose {m-2}}$ & $2{{2m-3}\choose{m-1}}$ & ${{2m-2}\choose m-2}$\\
\end{tabular}
\end{center}
\section{Future research}
Heuristic evidence suggests that for $m,n$ large, the density in equation \eqref{equ:kexceedancepaths}  
\begin{equation}
\label{equ:arcsine}
F(m,n,k,l/k)=\frac{k}{{{m+n}\choose m}}\#(m,n,k,l)
\end{equation}
is approximated by a beta distribution whose shape parameters depend on $m,n$ and $k$. The nature of this dependence could not be found. A particularly intriguing case is a symmetric $n=m=2k$, where it appears the distribution in equation \ref{equ:arcsine} converges, in probability, to the arcsine distribution as $k\to\infty$. This would imply that the distribution of random walks which start and end at the origin relate to the celebrated arcsine law of L\'evy. This is surprising because the arcsine law is for sufficiently long random walks with independent steps.  
\section{acknowledgement}
The second author would like to thank Steven Miller for several helpful suggestions during the preparation of this article.

C. Hoffman email: \email{{\href{mailto:hoffmace@flcc.edu}{hoffmace@flcc.edu}}}\\
C. Manack email: \email{{\href{mailto:cmanack1@fandm.edu}{cmanack1@fandm.edu}}}

\end{document}